\documentclass[11pt]{amsart}

\usepackage{amscd}
\usepackage{amsmath, amssymb}
\usepackage{amsfonts}
\newcommand{\de}{\partial}
\newcommand{\db}{\overline{\partial}}

\newcommand{\ddt}{\frac{\partial}{\partial t}}

\newcommand{\ddbar}{\sqrt{-1} \partial \overline{\partial}}

\newcommand{\ov}[1]{\overline{#1}}
\newcommand{\mn}{\sqrt{-1}}

\newcommand{\tr}[2]{\mathrm{tr}_{#1}{#2}}
\newcommand{\ti}[1]{\tilde{#1}}
\newcommand{\vp}{\varphi}

\newcommand{\ve}{\varepsilon}
\newcommand{\of}{\omega_{\mathrm{LF}}}

\newcommand{\bd}{\begin{enumerate}}
\newcommand{\ed}{\end{enumerate}}
\newcommand{\btheorem}{\begin{theorem}}
\newcommand{\etheorem}{\end{theorem}}
\newcommand{\bproposition}{\begin{proposition}}
\newcommand{\eproposition}{\end{proposition}}
\newcommand{\bdefinition}{\begin{definition}}
\newcommand{\edefinition}{\end{definition}}
\newcommand{\bcorollary}{\begin{corollary}}
\newcommand{\ecorollary}{\end{corollary}}
\newcommand{\bproof}{\begin{proof}}
\newcommand{\eproof}{\end{proof}}
\newcommand{\bremark}{\begin{remark}}
\newcommand{\eremark}{\end{remark}}
\newcommand{\eexample}{\end{example}}
\newcommand{\bexample}{\begin{example}}
\newcommand{\elemma}{\end{lemma}}
\newcommand{\blemma}{\begin{lemma}}
\newcommand{\ee}{\end{eqnarray*}}
\newcommand{\be}{\begin{eqnarray*}}

\newcommand{\glf}{g_{\mathrm{LF}}}

\newcommand{\Tlf}{T_{\mathrm{LF}}}

\numberwithin{equation}{section}
\renewcommand{\leq}{\leqslant}

\renewcommand{\le}{\leqslant}
\renewcommand{\ge}{\geqslant}

\begin{document}
\newtheorem{claim}{Claim}
\newtheorem{theorem}{Theorem}[section]
\newtheorem{lemma}[theorem]{Lemma}
\newtheorem{corollary}[theorem]{Corollary}
\newtheorem{proposition}[theorem]{Proposition}
\newtheorem{question}{question}[section]
\newtheorem{conjecture}[theorem]{Conjecture}

\theoremstyle{definition}
\newtheorem{remark}[theorem]{Remark}

\title{Inoue surfaces and the Chern-Ricci flow}
\author[S. Fang]{Shouwen Fang}
\address{College of Mathematical Sciences, Yangzhou University, Yangzhou 225002, P.R.China}
\author[V. Tosatti]{Valentino Tosatti}
\thanks{Supported in part by National Science Foundation of China grants 11401514 (S.F) and 11401023 (T.Z.), and  National Science Foundation grants DMS-1308988 (V.T.) and DMS-1406164 (B.W). The second-named author is supported in part by a Sloan Research Fellowship.}
\address{Department of Mathematics, Northwestern University, 2033 Sheridan Road, Evanston, IL 60208}
\author[B. Weinkove]{Ben Weinkove}
\address{Department of Mathematics, Northwestern University, 2033 Sheridan Road, Evanston, IL 60208}
\author[T. Zheng]{Tao Zheng}
\address{School of Mathematics and Statistics, Beijing Institute of Technology, Beijing 100081, P.R.China}
\begin{abstract} We investigate the Chern-Ricci flow, an evolution equation of Hermitian metrics, on Inoue surfaces.  These are  non-K\"ahler compact complex surfaces of type Class VII.  We show that, after an initial conformal change, the flow always collapses the Inoue surface to a circle at infinite time, in the sense of Gromov-Hausdorff.
\end{abstract}

\maketitle

\section{Introduction}

The Chern-Ricci flow is an evolution equation for Hermitian metrics on complex manifolds, which specializes to the K\"ahler-Ricci flow when the initial metric is K\"ahler.  It was  introduced by Gill \cite{G} in the setting of manifolds with vanishing first Bott-Chern class.  The second and third named authors  \cite{TW, TW2} investigated the flow on more general complex manifolds and  initiated a program to understand its behavior on all compact complex surfaces.  The results obtained in \cite{TW, TW2, TWY, G3, GS, Ni, ShW} are closely analogous to those for the K\"ahler-Ricci flow, and provide compelling evidence that the Chern-Ricci flow is a natural geometric flow on complex surfaces  whose behavior reflects the underlying geometry of these manifolds.

The Class VII surfaces are of particular interest, and indeed there is a well-known open problem to complete their classification.   Class VII surfaces are by definition compact complex surfaces with negative Kodaira dimension and first Betti number one.  In particular, they are non-K\"ahler.  It is natural to try to understand the behavior of the Chern-Ricci flow on these surfaces, with the long-term aim of extracting new topological or complex-geometric information (cf. \cite{STi}, where a different flow is considered).

The goal of this paper is to analyze the behavior of the Chern-Ricci flow on a family of
 already well-understood Class VII surfaces, known as Inoue surfaces.  Thanks to results in \cite{Bo, Ko, LYZ, T0}, we can characterize an Inoue surface $S$ as a Class VII surface with vanishing second Betti number and no holomorphic curves.   These surfaces were constructed and classified by Inoue in \cite{In}.  They come in three families, denoted by $S_M, S^+$ and $S^-$, and their universal cover is $\mathbb{C}\times H$, where $H$ is the upper half plane.
We will denote by $(z_1,z_2)$ the standard coordinates on $\mathbb{C}\times H$.  On any Inoue surface $S$ the standard Poincar\'e metric $\frac{\mn}{(\mathrm{Im}z_2)^2}dz_2\wedge d\ov{z}_2$ on $H$ descends to a closed semipositive real $(1,1)$ form on $S$, and a constant multiple of it (this constant depends on $S$) is then a closed real  $(1,1)$ form $\omega_{\infty}$ with
$$0 \le \omega_{\infty} \in - c_1^{\mathrm{BC}}(S),$$
where $c_1^{\mathrm{BC}}(S)$ is the first Bott-Chern class of $S$.
The $(1,1)$ form $\omega_{\infty}$ will play an important role in our results.

We consider the Chern-Ricci flow
\begin{equation} \label{NCRF}
\ddt{} \omega = - \textrm{Ric}(\omega) - \omega, \qquad \omega|_{t=0} = \omega_0,
\end{equation}
on $S$, starting at a Hermitian metric $\omega_0$.   Here $\textrm{Ric}(\omega)$ denotes the \emph{Chern-Ricci} form of the Hermitian metric $\omega = \sqrt{-1} g_{i\ov{j}}dz_i \wedge d\ov{z}_j$, defined by
$$\textrm{Ric}(\omega) = - \ddbar \log \det g.$$
 To be precise, (\ref{NCRF}) is the \emph{normalized} Chern-Ricci flow, which will give us a limit at infinity without the need to rescale.
Since the canonical bundle of $S$ is nef, it follows from the results of \cite{TW, TW2, TWY} that there exists a unique solution to (\ref{NCRF}) for all time. We are concerned with the behavior of the flow as $t \rightarrow \infty$.

In \cite{TW2}, an explicit solution to (\ref{NCRF}) was found on every Inoue surface. In fact, viewing Inoue surfaces as quotients $G/H$ where $G$ is a solvable Lie group \cite{Wa}, the metrics considered in \cite{TW2} are homogeneous, and it was later observed in \cite{L,LR} that the Chern-Ricci form of any homogeneous Hermitian metric depends only on the complex structure of the manifold (and not on the metric), which allows one to explicitly solve the Chern-Ricci flow starting at any homogeneous metric on an Inoue surface.

The explicit solutions constructed in \cite{TW2} (and in fact all the homogeneous solutions considered in \cite{L,LR}) have the property that they converge in the Gromov-Hausdorff sense to the circle $S^1$ with its standard metric. This reflects the structure of Inoue surfaces as bundles over $S^1$.  Note that this collapsing to $S^1$ is in striking contrast with the behavior of the K\"ahler-Ricci flow, where collapsed limits are always even-dimensional (cf. \cite{FZ, G2, ST, ST2, ST3,SW, TWY2}).

It is natural to conjecture that the limiting behavior of the explicit solutions \cite{TW2} holds for \emph{any} choice of initial metric.  The main result of this paper is that this is true for a large class of initial metrics.

 \begin{theorem} \label{maintheorem1}
Let $S$ be an Inoue surface, and let $\omega$ be any Hermitian metric on $S$.  Then there exists a Hermitian metric $\of = e^{\sigma} \omega$ in the conformal class of $\omega$ such that the following holds.

 Let $\omega(t)$ be the solution of the normalized Chern-Ricci flow (\ref{NCRF}) starting at a Hermitian metric of the form
  $$\omega_0 = \of+ \ddbar \rho>0.$$
  Then as $t \rightarrow \infty$,
$$\omega(t)\to \omega_\infty,$$
uniformly on $S$ and exponentially fast, where $\omega_{\infty}$ is the $(1,1)$ form defined above.  Moreover,
$$(S, \omega(t)) \rightarrow (S^1, d),$$
 in the Gromov-Hausdorff sense, where $d$ is the standard metric on the circle $S^1$ (of radius depending on $S$).
\end{theorem}

Thus we prove  that the Chern-Ricci flow collapses a Hermitian metric $\omega$ on the Inoue surface $S$ to a circle,
\emph{modulo an initial conformal change to $\omega$}.  In fact, we prove  more than this, since we allow our initial metric to be any metric in the $\partial \overline{\partial}$-class of $e^{\sigma}\omega$.


Our conformal change   is related to a holomorphic foliation structure on the Inoue surface $S$, which we now explain.
Every Inoue surface $S$ carries a holomorphic foliation $\mathcal{F}$ without singularities, whose leaves are the images of $\mathbb{C}\times \{z_2\}$, for every $z_2\in H$, under the quotient map $\mathbb{C}\times H \to S.$ In other words, the foliation $\mathcal{F}$ is defined by the subbundle of $TS$ which is the image of $T\mathbb{C}\oplus\{0\}$ under the differential of the quotient map. Every leaf of $\mathcal{F}$ is biholomorphic to $\mathbb{C}$ if $S$ is of type $S_M$, and biholomorphic to $\mathbb{C}^*$ if $S$ is of type $S^+$ or $S^-$ (see \cite{Br2}, for example). We will say that a Hermitian metric $\omega$ on $S$ is {\em flat along the leaves} if the restriction of $\omega$ to every leaf of $\mathcal{F}$ is a flat K\"ahler metric on $\mathbb{C}$ or $\mathbb{C}^*$. Equivalently, after pullback to the universal cover $\mathbb{C}\times H$, the metric restricted to every slice $\mathbb{C}\times \{z_2\}, z_2\in H,$ is flat. If $\omega$ has the further property that after pullback to $\mathbb{C}\times H$ its restriction to $\mathbb{C}\times \{z_2\}$ equals $c (\mathrm{Im}z_2) \mn dz_1\wedge d\ov{z}_1$ when $S$ is of type $S_M$ and
 $c \mn dz_1\wedge d\ov{z}_1$ when $S$ is of type $S^+$ or $S^-$, where $c>0$ is a constant, then we will say that $\omega$ is {\em strongly flat along the leaves}.  This is the property that we require of our metric $\of$ in the statement of Theorem \ref{maintheorem1}.

The point is that the assumption of being strongly flat along the leaves is not in fact very restrictive, since it can always be attained from any Hermitian metric by a conformal change:

\begin{proposition}\label{easy}
Given any Hermitian metric $\omega$ on an Inoue surface, there exists a smooth function $\sigma$ on $S$ such that $\of := e^{\sigma} \omega$ is strongly flat along the leaves.
\end{proposition}

There is a well-developed theory of simultaneous uniformization of leaves of holomorphic foliations by complex curves (see e.g. \cite{Br, Ca, Gh}), but it turns out that in our explicit situation the proof of this proposition becomes very simple.
As far as we know, it may well be the case that every Hermitian metric on an Inoue surface belongs to the $\de\db$-class of a Hermitian metric which is strongly flat along the leaves (see section \ref{sectionconj}, Question 1).

An interesting question is whether the uniform ($C^0$) convergence in Theorem \ref{maintheorem1} of $\omega(t)$ to $\omega_{\infty}$ can be strengthened to smooth $(C^{\infty})$ convergence.  As a partial result in this direction, we
 show that if the initial metric is of a more restricted type, then we can obtain convergence in $C^{\alpha}$ for $0< \alpha <1$. More precisely, Tricerri \cite{Tr} and Vaisman \cite{Va} constructed an explicit, homogeneous, Gauduchon metric $\omega_{\mathrm{TV}}$ on each Inoue surface, which is strongly flat along the leaves.  We consider initial metrics in the $\partial\ov{\partial}$-class of $\omega_{\mathrm{TV}}$ and prove:
 \begin{theorem} \label{mainprop}
Let $S$ be an Inoue surface, and let $\omega(t)$ be the solution of the normalized Chern-Ricci flow (\ref{NCRF}) starting at a Hermitian metric of the form
  $$\omega_0 = \omega_{\mathrm{TV}}+ \ddbar \rho>0.$$
  Then the metrics $\omega(t)$ are uniformly bounded in the $C^1$ topology, and as $t \rightarrow \infty$,
$$\omega(t)\to \omega_\infty,$$
in the $C^\alpha$ topology, for every $0<\alpha<1$.
\end{theorem}

A remark about the notation $\omega_{\mathrm{TV}}$ above.  In the proof of Theorem \ref{mainprop}, we will write $\omega_{\mathrm{T}}$ instead of $\omega_{\mathrm{TV}}$  for the Tricerri \cite{Tr} metric  on $S_M$, and $\omega_{\mathrm{V}}$ for the metric of Vaisman \cite{Va} on $S^+$ or $S^-$.

We note that the arguments in the proof of Theorem \ref{maintheorem1} are formally quite similar to the arguments of \cite{TWY} which considered the Chern-Ricci flow on elliptic bundles.
In fact one can use these same arguments  to extend the validity of the main theorem of \cite{TWY} to some Hermitian metrics on elliptic bundles which are not Gauduchon.\\

This paper is organized as follows. In section \ref{sectionSM} we prove Proposition \ref{easy} and Theorems \ref{maintheorem1} and \ref{mainprop} for the Inoue surfaces of type $S_M$. Those of type $S^+$ and $S^-$ are dealt with in section \ref{sectionSPM}. Lastly, in section \ref{sectionconj} we pose several conjectures and open problems on the Chern-Ricci flow.\\

\noindent
{\bf Acknowledgments. } The authors thank Xiaokui Yang for some helpful conversations.
  This work was carried out while the first and fourth named authors were visiting the Mathematics Department at Northwestern University and they thank the department for its hospitality and for providing a good academic environment.
\section{The Inoue surfaces $S_M$} \label{sectionSM}

We recall now the construction of the Inoue surfaces $S_M$.  They are defined by $S_M=( \mathbb{C} \times H)/\Gamma$, for $H$ the upper half plane, by a group of automorphisms $\Gamma$ of $\mathbb{C} \times H$ which we now describe.

Let $M \in \textrm{SL}(3, \mathbb{Z})$ be a matrix with one real eigenvalue $\lambda>1$ and two complex eigenvalues $\mu, \overline{\mu}$ (with $\mu \neq \overline{\mu}$).  Let $(\ell_1, \ell_2,
\ell_3)$ be a real eigenvector for $\lambda$ and $(m_1, m_2, m_3)$ an eigenvector for $\mu$.  Then $\Gamma$ is defined to be the group generated by the four automorphisms
$$f_0(z_1,z_2) = (\mu z_1, \lambda z_2),
 \ f_j(z_1,z_2) = (z_1+m_j, z_2+\ell_j), \ 1 \le j \le 3.$$
Here, we are writing $z_1= x_1+ \sqrt{-1}y_1$ for the coordinate on $\mathbb{C}$ and
 $z_2 = x_2 + \sqrt{-1}y_2$ for the coordinate on $H = \{ \textrm{Im}z_2=y_2>0 \}$.
The action of $\Gamma$ on $\mathbb{C} \times H$ is properly discontinuous with compact quotient.  For our calculations, we may work in a single compact fundamental domain for $S_M$ in $\mathbb{C} \times H$ using $z_1, z_2$ as local coordinates.  We may assume that $z_1$ and $z_2$ are uniformly bounded, and that $y_2$ is uniformly bounded below away from zero.

On $\mathbb{C} \times H$, define nonnegative $(1,1)$ forms $\alpha$ and $\beta$ by
$$\alpha = \frac{\sqrt{-1}}{4y_2^2} dz_2 \wedge d\ov{z}_2, \qquad \beta = \sqrt{-1} y_2 dz_1 \wedge d\ov{z}_1.$$
These forms are invariant under $\Gamma$ and hence descend to $(1,1)$ forms on the Inoue surface $S_M$, which we will denote by the same symbols.   In addition, $\alpha$ is $d$-closed.

 The metric $\omega_{\mathrm{T}}:= 4\alpha + \beta$ is called the Tricerri metric \cite{Tr}, and it was shown in \cite{TW2} that the solution to the normalized Chern-Ricci flow starting at $\omega_{\mathrm{T}}$ is given by
$$\omega(t) =  e^{-t} \beta + (1+3e^{-t}) \alpha     \longrightarrow  \alpha, \quad \textrm{as } t\rightarrow \infty.$$
Note that the first Bott-Chern class $c_1^{\textrm{BC}}(S_M)$ can be represented by
$$\textrm{Ric}(\omega_{\mathrm{T}}) = -\alpha \in c_1^{\textrm{BC}}(S_M).$$

The following lemma shows that every Hermitian metric $\omega$ on $S_M$ is conformal to a metric which is strongly flat along the leaves. Recall from \cite{In} that the map $\mathbb{C}\times H\to \mathbb{R}^+$ given by $(z_1,z_2)\mapsto y_2$ induces a smooth map $p:S_M\to S^1=\mathbb{R}^+/(x\sim \lambda x)$ which is a smooth $T^3$-bundle. Every leaf of $\mathcal{F}$ is contained in a $T^3$-fiber, and is in fact dense inside it.

\begin{lemma} \label{flatleaves}
A Hermitian metric $\of$ on $S_M$ is flat along the leaves if and only if
\begin{equation} \label{c}
\alpha \wedge \of = (p^*\eta) \alpha \wedge \beta,
\end{equation}
where $\eta:S^1\to\mathbb{R}^+$ is a smooth positive function. It is strongly flat along the leaves if and only if
\begin{equation} \label{cp}
\alpha \wedge \of = c \alpha \wedge \beta,
\end{equation}
where $c>0$ is a constant.
If $\omega$ is any Hermitian metric on $S_M$ and we define $\sigma \in C^{\infty}(S_M)$ by
$$e^\sigma =  \frac{\alpha \wedge \beta}{\alpha \wedge \omega},$$
then $\of = e^{\sigma} \omega$ satisfies (\ref{cp}) with $c=1$ and hence is strongly flat along the leaves.
\end{lemma}
\begin{proof} Let $\pi:\mathbb{C}\times H\to S_M$ be the quotient map. Then a Hermitian metric $\of$ on $S_M$ is flat along the leaves if and only if its pullback
$\pi^*\of$ is flat when restricted to any slice $\mathbb{C}\times\{z_2\}$, $z_2\in H$. We can write
$$\pi^*\of=a\mn dz_1\wedge d\ov{z}_1+b\mn dz_2\wedge d\ov{z}_2+c\mn dz_1\wedge d\ov{z}_2+\ov{c}\mn dz_2\wedge d\ov{z}_1,$$
so that
\begin{equation} \label{fr1}
\frac{\alpha\wedge\of}{\alpha\wedge\beta}=\frac{a}{y_2},
\end{equation}
and \eqref{c} is equivalent to
\begin{equation} \label{fr}
\frac{a}{y_2}=\pi^*p^*\eta.
\end{equation}
Note that the function $\pi^*p^*\eta$ depends only on $y_2$.
Since the restriction of $\pi^*\of$ to a slice $\mathbb{C}\times\{z_2\}$ equals $a\mn dz_1\wedge d\ov{z}_1$, and its Ricci curvature equals
$-\de_1\de_{\ov{1}}\log a,$ we conclude that if \eqref{c} holds then $\of$ is flat along the leaves.

Conversely, if $\of$ is flat along the leaves, then for each fixed $z_2\in H$ we have that
$$\de_1\de_{\ov{1}}\log \left(\frac{a}{y_2}\right)=0.$$
Thanks to \eqref{fr1} we see that the function $\log \frac{a}{y_2}$ on $\mathbb{C}\times H$ is $\Gamma$-invariant, hence bounded (because it is the pullback of a function from $S_M$). Therefore,
$\log \frac{a}{y_2}$ for $z_2$ fixed is a bounded harmonic function on $\mathbb{C}$, and so it must be constant. In other words, the ratio $\frac{\alpha\wedge\of}{\alpha\wedge\beta}$ is constant along each leaf of $\mathcal{F}$. Since every leaf is dense in the $T^3$ fiber which contains it, we conclude that $\frac{\alpha\wedge\of}{\alpha\wedge\beta}$ equals the pullback of a function from $S^1$.

On the other hand, it is now clear that $\omega$ is strongly flat along the leaves if and only if \eqref{cp} holds, or equivalently,
\begin{equation} \label{fr2}
\frac{a}{y_2}=c,
\end{equation}
where $c>0$ is a constant.  The last assertion of the lemma is immediate.
\end{proof}

Observe that in the coordinates $z_1, z_2$ above, the condition (\ref{cp}) is equivalent to
\begin{equation}\label{glf}
(\glf)_{1\ov{1}}  = cy_2,
\end{equation}
where we are writing $\of=\sqrt{-1} (\glf)_{i\ov{j}} dz_i \wedge d\ov{z}_j$.

Let now $\of$ be a metric which is strongly flat along the leaves, as in the set up of Theorem \ref{maintheorem1}.
We remark that in the following arguments this metric plays exactly the same role as the semi-flat metric $\omega_{{\rm flat}}$ considered in \cite{TWY} on elliptic bundles.
We will write the normalized Chern-Ricci flow as a parabolic complex Monge-Amp\`ere equation.  First, we define
\begin{equation} \label{omegahat}
\tilde{\omega}=\tilde{\omega}(t) = e^{-t} \of + (1-e^{-t}) \alpha>0,
\end{equation}
and write $\tilde{g}$ for the Hermitian metric associated to $\tilde{\omega}$.
We define a volume form $\Omega$ by
\begin{equation} \label{Omega}
\Omega = 2 \alpha \wedge \of = 2c \alpha \wedge \beta,
\end{equation}
for $c$ the constant defined by (\ref{cp}).  Observe that
$$\ddbar \log \Omega = \alpha,$$
by direct calculation using \eqref{Omega}.
It follows that the normalized Chern-Ricci flow (\ref{NCRF})  is equivalent to the equation
\begin{equation} \label{pma}
\ddt{} \varphi = \log \frac{e^t (\tilde{\omega} + \ddbar \varphi)^2}{\Omega} - \varphi, \ \tilde{\omega} + \ddbar \varphi>0, \ \varphi(0)= \rho.
\end{equation}
Namely, if $\varphi$ solves (\ref{pma}) then $\omega(t)= \tilde{\omega}+ \ddbar \varphi$ solves (\ref{NCRF}), as is readily checked.  Conversely, given a solution $\omega(t)$ of (\ref{NCRF}) we can find (cf. \cite{TW}, for example) a solution $\varphi=\varphi(t)$ of (\ref{pma}) with the property that $\omega(t)  = \tilde{\omega} + \ddbar \varphi$.

Let now $\varphi=\varphi(t)$ solve (\ref{pma}) and write $\omega= \omega(t) = \tilde{\omega} + \ddbar \varphi$ for the corresponding metrics along the normalized Chern-Ricci flow.
We first establish uniform estimates on the potential $\varphi$ and its time derivative $\dot{\varphi}$.   Given the choice of $\tilde{\omega}$ and $\Omega$, the proof is formally almost identical to
\cite[Lemma 3.6.3, Lemma 3.6.7]{SW} (see also \cite{ST, G2, FZ, TWY}).

\begin{lemma} \label{lemma1}
There exists a uniform positive constant $C$ such that on $S_M \times [0,\infty)$,
\begin{enumerate}
\item[(i)] $\displaystyle{|\varphi| \le C(1+t)e^{-t}}$.
\item[(ii)] $\displaystyle{| \dot{\varphi} | \le C}$.
\item[(iii)] $\displaystyle{ C^{-1} \tilde{\omega}^2 \le \omega^2 \le C \tilde{\omega}^2}$.
\end{enumerate}
\end{lemma}
\begin{proof} Since the arguments are so similar to those in \cite{SW, G2, TWY}, we will be brief.
For (i), the key observation is that, by the choice of $\tilde{\omega}$ and $\Omega$, we have
\begin{equation}
\left| e^t \log \frac{e^t \tilde{\omega}^2}{\Omega} \right| \le C',
\end{equation}
for uniform $C'$. To see this, note that from (\ref{omegahat}) and (\ref{Omega}),
$$e^t \log \frac{e^t \tilde{\omega}^2}{\Omega} = e^t \log \left( \frac{2 \alpha \wedge \of  + e^{-t} (\of^2 - 2 \of \wedge \alpha)}{2 \alpha \wedge \of} \right) = e^t \log (1+ O(e^{-t})),$$
and this is bounded.  If we define $Q=e^t \varphi - (C'+1)t$,  we see that if $\sup_{S_M \times [0,t_0]} Q= Q(x_0,t_0)$ for some $x_0 \in S_M$ and $t_0>0$, we have at that point,
$$0 \le \frac{\partial Q}{\partial t} \le e^t \log \frac{e^t \tilde{\omega}^2}{\Omega} -C' -1 \le -1,$$
a contradiction.  It follows that $\sup_{S_M} Q$ is bounded from above by its value at time $0$, giving   $\varphi \le C(1+t)e^{-t}$.  The lower bound is similar.

To prove (ii), choose a constant $C_0$ so that $C_0 \tilde{\omega} > \alpha$ for all $t \ge 0$.  Then compute, for $\Delta = g^{i\ov{j}} \partial_i \partial_{\ov{j}}$,
\[
\begin{split}
\left( \ddt{} - \Delta \right) \left(\dot{\varphi} - (C_0-1)\varphi \right) = {} & \tr{\omega}{(\alpha - \tilde{\omega})} + 1 - C_0\dot{\varphi} + (C_0-1) \tr{\omega}{(\omega- \tilde{\omega})} \\
< {} & 1- C_0 \dot{\varphi} + 2(C_0-1).
\end{split}
\]
It then follows from the maximum principle that $\dot{\varphi}$ is bounded from above.   For the lower bound of $\dot{\varphi}$,
\[
\begin{split}
\left( \ddt{} - \Delta \right) (\dot{\varphi} + 2\varphi ) =  {} & \tr{\omega}{(\alpha-\tilde{\omega})} + 1 + \dot{\varphi} - 2\tr{\omega}{(\omega-\tilde{\omega})} \\
\ge {} & \tr{\omega}{\tilde{\omega}} + \dot{\varphi} -3 \ge \frac{1}{C} e^{-(\dot{\varphi}+\varphi)/2} + \dot{\varphi} -3,
\end{split}
\]
for a uniform $C>0$, where for the last inequality we have used the geometric-arithmetic means inequality, the equation (\ref{pma}) and the fact that $e^t\tilde{\omega}^2$ and $\Omega$ are uniformly equivalent.  It follows from the maximum principle that $\dot{\varphi}$ is bounded from below.

Finally, (iii) is a consequence of (i), (ii) and the equation (\ref{pma}).
\end{proof}

Next,
we bound the torsion and curvature of the reference metrics $\tilde{g}$.  We will denote the Chern connection, torsion and curvature of $\tilde{g}$ by $\tilde{\nabla}$, $\tilde{T}$ and $\widetilde{\textrm{Rm}}$ respectively.

We lower an index of the torsion into the third subscript, so
$$\tilde{T}_{ij\ov{\ell}} := \tilde{g}_{k\ov{\ell}} \tilde{T}^k_{ij} = \partial_i \tilde{g}_{j\ov{\ell}} - \partial_j \tilde{g}_{i\ov{\ell}}.$$
  Writing $\Tlf$ for the torsion of the metric $\glf$, we see that since $\alpha$ is a closed form, we have
 \begin{equation} \label{torsiontilde}
 \tilde{T}_{ij\ov{\ell}} = e^{-t} (\Tlf)_{ij\ov{\ell}}.
 \end{equation}
We prove the following bounds on the torsion and curvature of $\tilde{g}$, which are analogous to those in \cite[Lemma 4.1]{TWY}.

\begin{lemma} \label{lemma2}
There exists a uniform constant $C$ such that
\begin{enumerate}
\item[(i)] $\displaystyle{| \tilde{T}|_{\tilde{g}} \le C}$.
\item[(ii)] $\displaystyle{ | \ov{\partial} \tilde{T} |_{\tilde{g}}+ | \tilde{\nabla} \tilde{T}|_{\tilde{g}} + | \widetilde{\emph{Rm}}|_{\tilde{g}} \le Ce^{t/2}}$.
\end{enumerate}
\end{lemma}
\begin{proof}
Using the holomorphic coordinates as described above and (\ref{glf}), we have
\begin{equation}  \label{tildegf}
\begin{split}
& \tilde{g}_{1\ov{1}} =  e^{-t} cy_2, \quad  \tilde{g}_{1\ov{2}} = e^{-t} (\glf)_{1\ov{2}}, \\
& \tilde{g}_{2\ov{2}} =  e^{-t} (\glf)_{2\ov{2}} + (1-e^{-t}) \frac{1}{4y_2^2}.
\end{split}
\end{equation}
Since we are working in a fixed chart with $y_2$  uniformly bounded (both from above and below away from zero), we have the estimates
\begin{equation} \label{tildege}
\begin{split}
& C^{-1} e^{-t} \le \tilde{g}_{1\ov{1}} \le C e^{-t}, \quad C^{-1} e^t \le \tilde{g}^{1\ov{1}} \le C e^t \\
& C^{-1} \le \tilde{g}_{2\ov{2}} \le C,  \quad C^{-1} \le \tilde{g}^{2\ov{2}} \le C \\
&  | \tilde{g}_{1\ov{2}}| \le Ce^{-t}, \quad | \tilde{g}^{1\ov{2}}| \le C.
\end{split}
\end{equation}
Observe that these estimates are identical to the local estimates given in \cite[Lemma 4.1]{TWY}, where here $z_1$ plays the role of the ``fiber'' direction and $z_2$ the role of the ``base'' direction.

As in \cite{TWY}, (i) follows from (\ref{torsiontilde}).  Indeed,
$$| \tilde{T}|^2_{\tilde{g}} = e^{-2t} \tilde{g}^{i\ov{j}} \tilde{g}^{k\ov{\ell}} \tilde{g}^{p\ov{q}} (\Tlf)_{ik\ov{q}} \ov{(\Tlf)_{j\ell \ov{p}}} \le C,$$
since the only terms involving the cube of $\tilde{g}^{1\ov{1}}$ vanish by the skew-symmetry of $(\Tlf)_{ik\ov{q}}$ in $i$ and $k$, and  by (\ref{tildege}) all other terms are bounded.

Next we have a bound on the $\tilde{g}$ norm of the Christoffel symbols $\tilde{\Gamma}^p_{ik}$ of the Chern connection of $\tilde{g}$:
\begin{equation} \label{Christoffel}
| \tilde{\Gamma}^p_{ik}|^2_{\tilde{g}} := \tilde{g}^{i\ov{j}} \tilde{g}^{k\ov{\ell}} \tilde{g}_{p\ov{q}} \tilde{\Gamma}^p_{ik} \ov{\tilde{\Gamma}^q_{j\ell}} = \tilde{g}^{i\ov{j}} \tilde{g}^{k\ov{\ell}} \tilde{g}^{p\ov{q}} \partial_i \tilde{g}_{k\ov{q}} \partial_{\ov{j}} \tilde{g}_{p\ov{\ell}} \le C e^t.
\end{equation}
This follows from (\ref{tildege}) and the fact that all terms $\partial_i  \tilde{g}_{k\ov{q}}$ are of order $O(e^{-t})$ unless $i=k=q=2$.  Note that the quantity $| \tilde{\Gamma}^p_{ik}|^2_{\tilde{g}}$ is only locally defined.

Then from (\ref{torsiontilde}),
$$| \ov{\partial} \tilde{T}|^2_{\tilde{g}} = | \tilde{\nabla}_{\ov{\ell}} \tilde{T}_{ij\ov{k}}|^2_{\tilde{g}} = e^{-2t} | \partial_{\ov{\ell}} (\Tlf)_{ij\ov{k}} - \ov{\tilde{\Gamma}^r_{\ell k}} (\Tlf)_{ij\ov{r}} |^2_{\tilde{g}} \le Ce^t,$$
where for the last inequality we have used (\ref{Christoffel}) and the estimates
$$| \partial_{\ov{\ell}} (\Tlf)_{ij\ov{k}}|^2_{\tilde{g}} \le Ce^{3t}, \quad |( \Tlf)_{ij\ov{r}}|^2_{\tilde{g}} \le Ce^{2t},$$
which follow from (\ref{tildege}) and the skew symmetry of torsion.  Similarly, we obtain
$$| \tilde{\nabla} \tilde{T} |^2_{\tilde{g}} \le C e^t,$$
and it remains to bound the curvature of $\tilde{g}$.

Recall that the curvature of the Chern connection of $\tilde{g}$ is given by
$$\tilde{R}_{i\ov{j}k\ov{\ell}} = - \partial_i \partial_{\ov{j}} \tilde{g}_{k\ov{\ell}} + \tilde{g}^{p\ov{q}} \partial_i \tilde{g}_{k\ov{q}} \partial_{\ov{j}} \tilde{g}_{p\ov{\ell}}.$$
Noting that from (\ref{tildegf}),
$$\partial_1 \partial_{\ov{1}} \tilde{g}_{1\ov{1}} =0, \quad |\partial_{i}\partial_{\ov{j}} \tilde{g}_{k\ov{\ell}}| \le Ce^{-t}, \quad \textrm{if } (i,j,k,\ell) \neq (2,2,2,2),$$
and
 $$\partial_1 \tilde{g}_{1\ov{1}}=0, \quad |\partial_k \tilde{g}_{i\ov{j}}| \le Ce^{-t},\quad \textrm{if } (i,j,k) \neq (2,2,2),$$
 we obtain using (\ref{tildege}),
 $$| \tilde{R}_{1\ov{1}1\ov{1}}| \le Ce^{-2t}.$$
Moreover, straightforward computations, similar to those in \cite[Lemma 4.1]{TWY}, give $| \tilde{R}_{2\ov{2}2\ov{2}}| \le C$, and
$$| \tilde{R}_{i\ov{j}k\ov{\ell}} | \le Ce^{-t}, \quad \textrm{for } i,j,k,\ell \ \textrm{not all equal}.$$
It then follows, by considering the various cases, again as in \cite{TWY}, that
$$|\widetilde{\textrm{Rm}}|^2_{\tilde{g}} = \tilde{g}^{i\ov{q}} \tilde{g}^{p\ov{j}} \tilde{g}^{k\ov{s}} \tilde{g}^{r\ov{\ell}} \tilde{R}_{i\ov{j}k\ov{\ell}} \ov{\tilde{R}_{q\ov{p}s\ov{r}}} \le Ce^t,$$
as required.
\end{proof}

We can then apply, almost verbatim, the arguments of \cite{TWY} to establish estimates for solutions of the normalized Chern-Ricci flow on $S_M$.  There are some minor differences, which we now discuss.  In \cite{TWY}, the reference $(1,1)$ forms  $\tilde{\omega}(t)$ are positive definite only for $t$ sufficiently large, whereas here they are positive definite for all $t\ge 0$.  In \cite{TWY} there is a global holomorphic surjective map $\pi$ from the manifold to a Riemann surface $(S, \omega_S)$, and the $(1,1)$ form $\pi^* \omega_S$ plays the role of $\alpha$.  In our case, $\alpha$ is not globally of this form.   Nevertheless, we will see that the arguments still go through without trouble.

\begin{theorem} \label{thmestimates1}
For $\varphi=\varphi(t)$ solving (\ref{pma}) on $S_M$, the following estimates hold.
\begin{enumerate}
\item[(i)]  There exists a uniform constant $C$ such that
$$C^{-1} \tilde{\omega} \le \omega(t) \le C \tilde{\omega}.$$
\item[(ii)] There exists a uniform constant $C$ such that the Chern scalar curvature  satisfies the bound
$$-C \le R \le C e^{t/2}.$$
\item[(iii)] For any $\eta, \sigma$ with $0<\eta<1/2$ and $0<\sigma<1/4$, there exists a constant $C_{\eta, \sigma}$ such that
$$-C_{\eta,\sigma} e^{-\eta t} \le \dot{\varphi} \le C_{\eta,\sigma}  e^{-\sigma t}.$$
\item[(iv)] For any $\ve>0$ with $0<\ve<1/8$ there exists a constant $C_{\ve}$ such that
$$(1-C_{\ve}e^{-\ve t}) \tilde{\omega} \le \omega(t) \le (1+C_{\ve}e^{-\ve t}) \tilde{\omega}.$$
\end{enumerate}
\end{theorem}
\begin{proof}
Since the proof is almost identical to the arguments of \cite{TWY}, given Lemma \ref{lemma1} and Lemma \ref{lemma2}, we give only a brief outline, pointing out the main differences.

To show that the metrics $\omega$ and $\tilde{\omega}$ are uniformly equivalent, we apply the maximum principle to the quantity
$$Q= \log \tr{\tilde{\omega}}{\omega} - Ae^{t/2} \varphi + \frac{1}{\tilde{C} + e^{t/2} \varphi},$$
where $\tilde{C}$ is a uniform constant chosen so that $\tilde{C}+e^{t/2} \varphi \ge 1$, and for $A$ a uniform large constant.  A key point is that by Lemma \ref{lemma1}, the quantity $e^{t/2} \varphi$ is uniformly bounded.  The term $1/(\tilde{C} + e^{t/2} \varphi)$, of Phong-Sturm type \cite{PS}, is added to deal with torsion terms arising in the computation (cf. \cite{TW}).  Arguing as in \cite[Theorem 5.1]{TWY}, choosing $A$ sufficiently large,  $Q$, and hence $\tr{\tilde{\omega}}{\omega}$, is uniformly bounded from above.  Then from part (iii) of Lemma \ref{lemma1}, we obtain
$C^{-1} \tilde{\omega} \le \omega \le C\tilde{\omega}.$

The lower bound $R \ge -C$ follows immediately from the evolution equation
$$\frac{\partial R}{\partial t} = \Delta R + |\textrm{Ric}|^2+R,$$
and the minimum principle.  The upper bound of $R$ requires a number of preliminary estimates, which we now explain.

As in \cite[Lemma 6.2]{TWY}, we have the evolution inequalities
\begin{equation} \label{ee1}
\begin{split}
\left( \ddt{} - \Delta \right) \tr{\tilde{\omega}}{\omega} \le {} & - C^{-1} | \tilde{\nabla} g|^2_g + Ce^{t/2} \\
\left( \ddt{} - \Delta \right) \tr{\omega}{\alpha} \le {} & | \tilde{\nabla} g|^2_g - C^{-1} | \nabla \tr{\omega}{\alpha}|^2_g + Ce^{t/2},
\end{split}
\end{equation}
and it follows that there exist uniform constants $C_0, C_1>0$ such that
\begin{equation} \label{nice}
\left( \ddt{} - \Delta \right) (\tr{\omega}{\alpha} + C_0 \tr{\tilde{\omega}}{\omega}) \le  - | \tilde{\nabla} g|^2_g - C_1^{-1} | \nabla \tr{\omega}{\alpha}|^2_g + Ce^{t/2}.
\end{equation}
Indeed, the first evolution inequality in (\ref{ee1}) follows easily from the evolution equation of $\tr{\tilde{\omega}}{\omega}$.  The second inequality of (\ref{ee1}) is a parabolic Schwarz Lemma argument as in \cite{Ya, ST}.  An important difference to note here is that we do not have a global holomorphic map from $S_M$ to a lower dimensional complex manifold.  However, we do have a \emph{locally defined} holomorphic map $f$ from a holomorphic chart in $S_M$ to the upper half plane $H$ with the property that $\alpha = f^* \omega_P$, for $\omega_P$ the Poincar\'e metric
$$\omega_P =  \frac{\sqrt{-1}}{4y_2^2} dz_2 \wedge d\ov{z}_2,$$
on $H$.  Since the parabolic Schwarz Lemma computation is purely local, we obtain the second inequality of (\ref{ee1}) exactly as in \cite{TWY}.

Now consider the bounded quantity
$$u = \varphi + \dot{\varphi},$$
which has the property that $-\Delta u = R + \tr{\omega}{\alpha} \ge R$.  Our goal is to bound $-\Delta u$ from above by $Ce^{t/2}$.  First,
using a Cheng-Yau \cite{CY} type argument (cf. \cite{SeT, ST}), we apply the maximum principle to
$$Q_1 = \frac{| \nabla u|^2_g}{A-u} + C_1 (\tr{\omega}{\alpha} + C_0 \tr{\tilde{\omega}}{\omega}),$$
for $A$ and $C_1$ chosen sufficiently large, we obtain the estimate
$$| \nabla u|^2_g \le Ce^{t/2},$$
exactly as in \cite[Proposition 6.3]{TWY} (replacing $\omega_S$ with $\alpha$, wherever it occurs).  A straightforward computation then gives
\begin{equation} \label{nablau}
\left( \frac{\partial}{\partial t} - \Delta \right) | \nabla u|^2_g \le - \frac{1}{2} | \nabla \ov{\nabla} u|^2_g - | \nabla \nabla u|^2_g + | \nabla \tr{\omega}{\alpha}|^2_g + | \tilde{\nabla}g |^2_g + Ce^t.
\end{equation}
On the other hand, we have
\begin{equation} \label{Deltau}
\left( \ddt{} - \Delta \right) (- \Delta u) \le  2  | \nabla \ov{\nabla} u|^2_g - \Delta u + Ce^{t/2} + |\tilde{\nabla}g|^2_g - C^{-1} | \nabla \tr{\omega}{\alpha}|^2_g.
\end{equation}
Combining (\ref{nice}), (\ref{nablau}), (\ref{Deltau}), we obtain, for $C_1$ large,
$$\left( \ddt{} - \Delta \right) (-\Delta u + 6| \nabla u|^2_g + C_1 (\tr{\omega}{\alpha} + C_0 \tr{\tilde{\omega}}{\omega})) \le - | \nabla \ov{\nabla} u|^2_g -\Delta u + Ce^t,$$
and it follows from the maximum principle that $-\Delta u \le Ce^{t/2}$, giving the required upper bound for the Chern scalar curvature
$$ R \le Ce^{t/2}.$$

Next we use our bounds on $R$ together with Lemma \ref{lemma1} and
$$\ddt{} \dot{\varphi} = -R -1 -\dot{\varphi},$$
to get, exactly as in \cite[Lemma 6.4]{TWY} (cf. \cite{SW}), the bound (iii) on $\dot{\varphi}$.

For (iv), we first have, as in \cite[Lemma 7.3]{TWY}, that
$$\left( \ddt{}  - \Delta \right) \tr{\omega}{\tilde{\omega}} \le Ce^{t/2} - C^{-1} |\tilde{\nabla} g|^2_g.$$
Using this, and the bounds we have now established for $\varphi$ and $\dot{\varphi}$, we consider the quantities
$$e^{\ve t} (\tr{\omega}{\tilde{\omega}}-2) - e^{\delta t}\varphi, \quad e^{\ve t} (\tr{\tilde{\omega}}{\omega}-2) -e^{\delta' t}\varphi,$$
for $0< \ve<1/4$ and $\delta, \delta'>0$ chosen carefully.  The maximum principle argument of \cite[Proposition 7.3]{TWY} shows that
$$\tr{\omega}{\tilde{\omega}} - 2\le Ce^{-\ve t}, \quad \tr{\tilde{\omega}}{\omega}-2 \le Ce^{-\ve t}.$$
Part (iv) now follows from an elementary argument \cite[Lemma 7.4]{TWY}.
\end{proof}

Given the estimate (iv) of Theorem \ref{thmestimates1}, we can now establish the Gromov-Hausdorff convergence of the metrics $\omega(t)$.

\begin{proof}[Proof of Theorem \ref{maintheorem1} for the Inoue surfaces $S_M$]
We set $\omega_\infty=\alpha$, which represents $-c_1^{\mathrm{BC}}(S_M)$, and by definition $\ti{\omega}(t)\to\alpha$ uniformly and exponentially fast as $t\to\infty$. Thanks to Theorem \ref{thmestimates1} (iv), the same convergence holds for $\omega(t)$, and it remains to determine the Gromov-Hausdorff limit of $(S_M,\omega(t))$.
The argument is a slight modification of \cite[Theorem 5.1]{TW2} (see also \cite[Lemma 9.1]{TWY}).
Let $p:S_M\to S^1$ be the bundle projection map, and denote by $T_y=p^{-1}(y)$  the $T^3$-fiber over $y\in S^1$. Fix $\ve>0$, and let $L_t$ be the length of a curve in $S_M$ measured with respect to $\omega(t)$ and by $d_t$ the induced distance function on $S_M$. Let also $L_\infty, d_\infty$ be the length and distance functions of the degenerate metric $\alpha$ on $S_M$ and let $L,d$ be those of the standard metric on $S^1$ with diameter $(\log\lambda)/\sqrt{2}\pi$. Let $F=p:S_M\to S^1$ and let $G:S^1\to S_M$ be the map defined by sending every point $y\in S^1$ to some chosen point in $S_M$ on the fiber $T_y$. The map $G$ will in general be discontinuous, and it satisfies $F\circ G=\mathrm{Id}$, so
\begin{equation}\label{gh1}
d(y,F(G(y)))=0.
\end{equation}
The limiting semipositive form $\alpha$ has kernel equal to $\mathcal{F}$, and since each leaf of $\mathcal{F}$ is dense in a $T^3$-fiber, we conclude that $d_\infty(x,y)=0$ for all $x,y\in S_M$ with $p(x)=p(y)$. Since $\omega(t)$ converges uniformly to $\alpha$ as $t\to\infty$, we see that for any $x\in S_M$ and for all $t$ large we have
\begin{equation}\label{gh2}
d_t(x,G(F(x)))\leq \ve.
\end{equation}
Next, a simple calculation (see \cite[Theorem 5.1]{TW2}) shows that for any curve $\gamma$ in $S_M$ we have
$$L(F(\gamma))=L_\infty(\gamma).$$
Therefore, given two points $x,y\in S_M$ let $\gamma$ be a minimizing geodesic for the metric $\omega(t)$ joining them, and get
\begin{equation}\label{gh3}
d(F(x),F(y))\leq L(F(\gamma))=L_\infty(\gamma)\leq L_t(\gamma)+\ve=d_t(x,y)+\ve,
\end{equation}
for all $t$ large. Obviously this also implies that
\begin{equation}\label{gh4}
d(x,y)\leq d_t(G(x),G(y))+\ve,
\end{equation}
for all $x,y\in S^1,$ and all $t$ large. Lastly, given $x,y\in S_M$, let $\gamma$ be a minimizing geodesic in $S^1$ connecting $F(x)$ and $F(y)$, and let $\tilde{\gamma}$ be a lift of the curve $\gamma$ starting at $x$, i.e. $\ti{\gamma}$ is a curve in $S_M$ with $F(\ti{\gamma})=\gamma$ and initial point $x$. This lift can always be constructed because $F$ is a fiber bundle. We then concatenate $\ti{\gamma}$ with a curve $\ti{\gamma}_1$ contained in the fiber $T_{F(y)}$ joining the endpoint of $\ti{\gamma}$ with $y$, and obtain a curve $\hat{\gamma}$ in $S_M$ joining $x$ and $y$. By construction we have
\begin{equation}\label{gh5}
d_t(x,y)\leq L_t(\hat{\gamma})=L_t(\ti{\gamma})+L_t(\ti{\gamma}_1)\leq L_\infty(\ti{\gamma})+2\ve = L(\gamma)+2\ve=d(F(x),F(y))+2\ve,
\end{equation}
for all $t$ large. This also implies
\begin{equation}\label{gh6}
d_t(G(x),G(y))\leq d(x,y)+2\ve,
\end{equation}
for all $x,y\in S^1,$ and all $t$ large. Combining \eqref{gh1}, \eqref{gh2}, \eqref{gh3}, \eqref{gh4}, \eqref{gh5}, \eqref{gh6}, we conclude that $(S_M,\omega(t))$ converges to $(S^1,d)$ in the Gromov-Hausdorff sense.
\end{proof}

We now turn to Theorem \ref{mainprop}. Our main claim is that if the initial metric is in the $\de\db$-class of the Tricerri metric $\omega_{\textrm{T}}= 4\alpha + \beta$, then we have the estimate
\begin{equation}\label{c1}
|\nabla_{\mathrm{T}} g |_{g_{\mathrm{T}}}\leq C,
\end{equation}
for all $t$ large, where $\nabla_{\mathrm{T}}$ is the covariant derivative of $g_{\mathrm{T}}$. Clearly, given the results proved in Theorem \ref{maintheorem1}, we see that Theorem \ref{mainprop} follows immediately from \eqref{c1}.

To prove \eqref{c1} the first step is the following improvement of Lemma \ref{lemma2}.  Here, $\omega_{\mathrm{T}}$ plays the role of $\of$.  In particular, $\alpha \wedge \omega_{\mathrm{T}}  = \alpha \wedge \beta$, so that the analogue of (\ref{cp}) holds with $c=1$.  We have
$$\ti{\omega}=e^{-t}\omega_{\mathrm{T}}+(1-e^{-t})\alpha.$$
These are completely explicit Hermitian metrics on $S_M$.  They
 have the property that $\omega=\ti{\omega}+\ddbar\vp$, where $\varphi$ solves  (\ref{pma}) (with $\Omega = 2\alpha \wedge \beta$).

\pagebreak[3]
\begin{lemma} \label{lemma2bis}
There exists a uniform constant $C$ such that
\begin{enumerate}
\item[(i)] $\displaystyle{| \tilde{T}|_{\tilde{g}} \le C}$.
\item[(ii)] $\displaystyle{ | \ov{\partial} \tilde{T} |_{\tilde{g}}+ | \tilde{\nabla} \tilde{T}|_{\tilde{g}} + | \widetilde{\emph{Rm}}|_{\tilde{g}} \le C}$.
\item[(iii)] $\displaystyle{ \ti{\nabla}\widetilde{\emph{Rm}}=0, \ov{\ti{\nabla}} \ov{\ti{\nabla}} \tilde{T}=0,  \tilde{\nabla} \ov{\ti{\nabla}}\tilde{T}=0}$.
\end{enumerate}
\end{lemma}
\begin{proof}
Part (i) was proved in Lemma \ref{lemma2}. The rest of the lemma consists of straightforward, if somewhat tedious, calculations.
Using the holomorphic coordinates as described earlier, we have
\[\tilde{g}_{1\ov{1}} =  e^{-t} y_2, \quad  \tilde{g}_{1\ov{2}} = 0, \quad \tilde{g}_{2\ov{2}} =  \frac{1+3e^{-t}}{4y_2^2},\]
\[\ti{g}^{1\ov{1}}=\frac{e^t}{y_2}, \quad \tilde{g}^{1\ov{2}}=0,\quad \ti{g}^{2\ov{2}}=\frac{4y_2^2}{1+3e^{ -t}}.\]
From this it follows that the only nonzero Christoffel symbols are
\begin{equation}\label{christ}
\ti{\Gamma}_{21}^1=-\frac{\mn}{2y_2},\quad \ti{\Gamma}_{22}^2=\frac{\mn}{y_2},
\end{equation}
and note that these do not depend on $t$.
The only nonzero components of the torsion $\tilde{T}$ and its  derivative $\db\ti{T}$ are
$$\ti{T}_{12}^1=-\ti{T}_{21}^1=\frac{\mn}{2y_2}, \quad \de_{\ov{2}} \ti{T}_{12}^1=-\de_{\ov{2}} \ti{T}_{21}^1=\frac{1}{4y_2^2}.$$
It follows that  $| \ov{\partial} \tilde{T} |_{\tilde{g}}\leq C$. The only nonzero components of $\ti{\nabla}\ti{T}$ are
$$\ti{\nabla}_{2}\ti{T}_{12}^1=-\ti{\nabla}_{2}\ti{T}_{21}^1=\frac{1}{4y_2^2},$$
and $| \tilde{\nabla} \tilde{T}|_{\tilde{g}}\leq C$ follows. The only nonzero components of $\widetilde{\mathrm{Rm}}$ are
$$\ti{R}_{2\ov{2}2\ov{2}}=-\frac{1+3e^{-t}}{8y_2^4}, \quad \ti{R}_{2\ov{2}1\ov{1}}=\frac{e^{-t}}{4y_2},$$
and from this we obtain $| \widetilde{\mathrm{Rm}}|_{\tilde{g}} \leq C$, finishing the proof of (ii).

For (iii), the only terms which are not immediately seen to be zero are:
$$\ti{\nabla}_2 \ti{R}_{2\ov{2}2\ov{2}}=\de_2 \ti{R}_{2\ov{2}2\ov{2}} -2\ti{\Gamma}_{22}^2 \ti{R}_{2\ov{2}2\ov{2}}=-2\mn \frac{1+3e^{-t}}{8y_2^5}+2\frac{\mn}{y_2}\cdot\frac{1+3e^{-t}}{8y_2^4}=0,$$
$$\ti{\nabla}_2\ti{R}_{2\ov{2}1\ov{1}}=\de_2 \ti{R}_{2\ov{2}1\ov{1}}-\ti{\Gamma}_{22}^2\ti{R}_{2\ov{2}1\ov{1}}-\ti{\Gamma}_{21}^1\ti{R}_{2\ov{2}1\ov{1}}=
\frac{\mn e^{-t}}{8y_2^2}-\frac{\mn e^{-t}}{8y_2^2}=0,$$
$$\ti{\nabla}_{\ov{2}}\ti{\nabla}_{\ov{2}}\ti{T}_{12}^1=\de_{\ov{2}}\de_{\ov{2}}\ti{T}_{12}^1-\ov{\ti{\Gamma}_{22}^2}\de_{\ov{2}}\ti{T}_{12}^1=-\frac{\mn}{4y_2^3}+\frac{\mn}{4y_2^3}=0,$$
$$\ti{\nabla}_{2}\ti{\nabla}_{\ov{2}}\ti{T}_{12}^1=\de_{2}\de_{\ov{2}}\ti{T}_{12}^1-\ti{\Gamma}_{21}^1\de_{\ov{2}}\ti{T}_{12}^1-\ti{\Gamma}_{22}^2\de_{\ov{2}}\ti{T}_{12}^1+\ti{\Gamma}^1_{21}\de_{\ov{2}}\ti{T}_{12}^1=
\frac{\mn}{4y_2^3}-\frac{\mn}{4y_2^3}=0,$$
and this completes the proof of the lemma.
\end{proof}

The estimates of Theorem \ref{thmestimates1} imply in particular that the metrics $\tilde{g}$ and $g$ are uniformly equivalent.  The key additional estimate we obtain in this case is the following Calabi-type ``third order'' estimate.

\begin{proposition}\label{cal}
We have
\begin{equation}\label{c11}
|\ti{\nabla} g |_{\ti{g}}\leq C.
\end{equation}
\end{proposition}
\begin{proof}
This is analogous to the Calabi-type estimates established in \cite{ShW}, \cite[Section 8]{TWY} (see also \cite{PSS}). The result that we need is not exactly contained in \cite{ShW}, but the adjustments needed are minimal, so we will only indicate what changes in this case.  We bound the quantity $S := |\ti{\nabla} g|_{g}^2=| \Psi|_g^2$, where $\Psi = \Gamma - \tilde{\Gamma}$.  The quantity $S$ is equivalent to
$|\ti{\nabla} g |^2_{\ti{g}}$.

First, compared to the setup in \cite{ShW}, the Chern-Ricci flow that we consider now is normalized.  However, this only introduces new negligible terms. Second, the reference metrics $\ti{g}$ now depend on $t$, while the reference metric $\hat{g}$ in \cite{ShW} is fixed.  But the Christoffel symbols  $\tilde{\Gamma}$ are independent of $t$ and so this does not introduce any new terms in the evolution of $S$.

The time dependence of $\ti{g}$ \emph{does} introduce a
new term in the evolution of $\tr{\ti{\omega}}{\omega}$, but this is  easily seen to be harmless (cf. \cite[(6.1)]{TWY}). Thanks to \cite[Remark 3.1]{ShW}, the bounds proved in Lemma \ref{lemma2bis} are then sufficient to bound $S$ and thus obtain \eqref{c11}.
\end{proof}

\begin{proof}[Proof of Theorem \ref{mainprop} for the Inoue surfaces $S_M$]
We can now complete the proof of \eqref{c1}, which implies Theorem \ref{mainprop}.
The first observation is that the Christoffel symbols $\ti{\Gamma}$ of $\ti{g}$, which we computed in \eqref{christ}, do not depend on $t$, and so
the covariant derivative $\ti{\nabla}$ with respect to the metric $\ti{g}$ equals the covariant derivative $\nabla_{\mathrm{T}}$ with respect to $\ti{g}|_{t=0}=g_{\mathrm{T}}$.
The second observation is that $\ti{g}\leq Cg_{\mathrm{T}}$ for a uniform constant $C$. Therefore \eqref{c1} follows immediately from \eqref{c11}.
\end{proof}

\section{The Inoue surfaces $S^+$ and $S^-$}\label{sectionSPM}

We recall now the construction of the Inoue surfaces $S^+$ from \cite{In} (see also \cite[Section 6]{TW2}).  Consider $N=(n_{ij}) \in \textrm{SL}(2, \mathbb{Z})$ with two real eigenvalues $\alpha, 1/\alpha$ with $\alpha>1$, and let $(a_1, a_2)$ and $(b_1, b_2)$ be real eigenvectors corresponding to $\alpha$ and $1/\alpha$ respectively.

Fix integers $p,q,r$, with $r\neq 0$, and a complex number ${\bf t}$.
With this data, let $(c_1,c_2)\in \mathbb{R}^2$ solve
the linear equation
$$(c_1,c_2)=(c_1,c_2)\cdot N^{t} + (e_1,e_2) + \frac{b_1a_2-b_2a_1}{r}(p,q),$$
where
$$e_i=\frac{1}{2} n_{i1}(n_{i1}-1)a_1b_1 + \frac{1}{2} n_{i2}(n_{i2}-1)a_2b_2 + n_{i1}n_{i2} b_1a_2,
\quad i=1,2.$$
Now let $\Gamma$ be the group of automorphisms of $\mathbb{C} \times H$ generated by
$$f_0(z_1,z_2)= (z_1 + {\bf t}, \alpha z_2),$$
$$f_j(z_1,z_2)= (z_1 + b_j z_2 + c_j, z_2+a_j), \quad \textrm{for } j=1,2,$$
$$\textrm{and} \quad f_3(z_1,z_2)=  \left(z_1+\frac{b_1a_2-b_2a_1}{r}, z_2\right).$$
The group $\Gamma$ acts properly discontinuously on $\mathbb{C} \times H$ with compact quotient $S^{+}=(\mathbb{C}\times H)/\Gamma$.
For each such $N, p,q,r$ and $\mathbf{t}$,  $S^{+}$ is an Inoue surface.  As in Section \ref{sectionSM}, we may work in a local holomorphic coordinate chart with $z_1$ and $z_2$  uniformly bounded, and  $y_2$ uniformly bounded below away from zero.

We will not give the precise definition of Inoue surfaces of type $S^{-}$ (see e.g. \cite{In, TW2}) because their only property that we will need is that every such surface has an unramified double cover which is an Inoue surface of type $S^{+}$, and we can easily derive the statements of our theorems on $S^-$ from the corresponding statements on $S^+$.

Therefore, going back to $S^+$, since $\alpha>1$, we can write $\mathrm{Im}{\bf t}=m\log \alpha$ for some $m\in\mathbb{R}$. Then the  $(1,0)$ forms on $H\times\mathbb{C}$,
$$\frac{1}{y_2}dz_2,\quad dz_1 -\frac{y_1-m\log y_2}{y_2}dz_2,$$
where $z_1=x_1+\mn y_1$ and  $z_2=x_2+\mn y_2$, are invariant under the $\Gamma$-action, and so descend to $S^{+}$.  Define $(1,1)$ forms $\alpha', \gamma$ on $S^+$ by
$$\alpha' = \frac{1}{2y_2^2} \sqrt{-1} dz_2 \wedge d\ov{z}_2,$$
$$\gamma = \sqrt{-1} \left( dz_1 -\frac{y_1-m\log y_2}{y_2}dz_2 \right) \wedge \left( d\ov{z}_1 -\frac{y_1-m\log y_2}{y_2}d\ov{z}_2 \right).$$
Then
$$\omega_{\mathrm{V}} =  2\alpha' + \gamma$$
is a Hermitian metric, originally introduced by Vaisman \cite{Va} (and Tricerri in the case $m=0$ \cite{Tr}).  Moreover,
$$\textrm{Ric}(\omega_{\mathrm{V}}) = - \alpha' \in c_1^{\textrm{BC}}(S^+).$$
It is shown in \cite{TW} that the solution of the normalized Chern-Ricci flow starting at $\omega_{\mathrm{V}}$ is given by
$$\omega(t) = e^{-t} \gamma + (1+e^{-t})\alpha' \longrightarrow \alpha', \quad \textrm{as } t \rightarrow \infty.$$

Now,  as in Lemma \ref{flatleaves}, we see that a metric $\of$ on $S^+$ is strongly flat along the leaves if and only if
\begin{equation} \label{alphapwb}
\alpha' \wedge \of = c \alpha' \wedge \gamma,
\end{equation}
for some positive constant $c$.   Indeed, this easily follows from the fact that in the coordinates $z_1, z_2$ above, the condition (\ref{alphapwb}) is equivalent to
\begin{equation} \label{glf2}
(\glf)_{1\ov{1}} = c.
\end{equation}
Moreover, given any Hermitian metric $\omega$ on $S^+$, we can produce $\of = e^{\sigma} \omega$ satisfying this condition, for suitable choice of $\sigma$.

Given this set-up, and what we have already proved, it is straightforward to complete the proof of Gromov-Hausdorff convergence for the manifolds $S^+$.

\begin{proof}[Proof of Theorem \ref{maintheorem1} for the Inoue surfaces $S^+$ and $S^-$]  The proof for the surfaces $S^+$ is almost identical to the case of $S_M$.  Indeed, we simply replace $\alpha$ by $\alpha'$ and note that the condition (\ref{glf}) has been replaced by the (even simpler) condition (\ref{glf2}).  The same estimates as in Lemma \ref{lemma2} hold almost verbatim, and similarly for Theorem \ref{thmestimates1}.  To obtain the Gromov-Hausdorff convergence to $(S^1, d)$, we apply the same argument as in the proof of Theorem \ref{maintheorem1}, the only difference being that now $S^+$ is a fiber bundle over $S^1$ with fiber a certain compact $3$-manifold (not a torus), but the leaves of $\mathcal{F}$ are still dense in these fibers (see e.g. \cite[Lemma 6.2]{TW2}), and the argument goes through as before.

Finally, the result for $S^-$ follows from the fact every such surface has an unramified double cover which is an Inoue surface of type $S^+$ (cf. \cite[Section 7]{TW2}).
\end{proof}

We finally discuss Theorem \ref{mainprop} for the Inoue surfaces $S^+$ and $S^-$. As before, the case of $S^-$ is immediately reduced to the case of $S^+$. The proof for the surfaces $S^+$ is similar to the case of $S_M$, but there are some differences. We now assume that the initial metric is in the $\de\db$-class of the Vaisman metric $\omega_{\textrm{V}}= 2\alpha' + \gamma$, and our goal is to prove the estimate
\begin{equation}\label{c1new}
|\nabla_{\mathrm{V}} g |_{g_{\mathrm{V}}}\leq C,
\end{equation}
for all $t$ large, where $\nabla_{\mathrm{V}}$ is the covariant derivative of $g_{\mathrm{V}}$. Again, this immediately implies the conclusion of Theorem \ref{mainprop} in this setting.
As before, we have the explicit reference metrics on $S^+$
$$\ti{\omega}=e^{-t}\omega_{\mathrm{V}}+(1-e^{-t})\alpha'.$$
The analog of Lemma \ref{lemma2bis} is now:
\begin{lemma} \label{lemma2ter}
There exists a uniform constant $C$ such that
\begin{enumerate}
\item[(i)] $\displaystyle{| \tilde{T}|_{\tilde{g}} \le C}$.
\item[(ii)] $\displaystyle{ | \ov{\partial} \tilde{T} |_{\tilde{g}}+ | \tilde{\nabla} \tilde{T}|_{\tilde{g}} + | \widetilde{\emph{Rm}}|_{\tilde{g}} \le C}$.
\item[(iii)] $\displaystyle{ |\ti{\nabla}\widetilde{\emph{Rm}}|_{\ti{g}}+|\ov{\ti{\nabla}}\ov{ \ti{\nabla}} \tilde{T}|_{\ti{g}}+|\tilde{\nabla} \ov{\ti{\nabla}}\tilde{T}|_{\ti{g}}\leq C}$.
\end{enumerate}
\end{lemma}

\begin{proof}
Again, item (i) was proved in Lemma \ref{lemma2}. The rest of the lemma follows from a long, direct calculation. The situation is considerably more complicated than in Lemma \ref{lemma2bis} because now very few terms vanish.

Using the holomorphic coordinates as described earlier, we have
\[\tilde{g}_{1\ov{1}} =  e^{-t}, \quad  \tilde{g}_{1\ov{2}} = -e^{-t}\frac{y_1-m\log y_2}{y_2}, \quad \tilde{g}_{2\ov{2}} =  \frac{1+e^{-t}+2e^{-t}(y_1-m\log y_2)^2}{2y_2^2},\]
\[\ti{g}^{1\ov{1}}=O(e^t), \quad \tilde{g}^{1\ov{2}}=O(1),\quad \ti{g}^{2\ov{2}}=O(1).\]
None of the Christoffel symbols of $\ti{g}$ vanishes, but their components in these coordinates are readily seen to be uniformly bounded as $t$ approaches infinity, which shows that
\begin{equation}\label{c13}
|\ti{\Gamma}-\Gamma_{\mathrm{V}}|_{g_{\mathrm{V}}}\leq C,
\end{equation}
where $\Gamma_{\mathrm{V}}$ are the Christoffel symbols of $g_{\mathrm{V}}$.
For later use, we mention explicitly that
\begin{equation}\label{christ2}
\begin{split}
\ti{\Gamma}_{11}^2&=\mn\frac{y_2}{1+e^t}=O(e^{-t}),\\
\ti{\Gamma}_{12}^2&=-\mn\frac{y_1-m\log y_2}{1+e^t}=O(e^{-t}),\\
\ti{\Gamma}_{21}^2&=-\mn\frac{m}{1+e^t}-\mn\frac{y_1-m\log y_2}{1+e^t}=O(e^{-t}).
\end{split}
\end{equation}
The torsion of $\ti{g}$ is given by
$$\ti{T}_{12}^1=-\ti{T}_{21}^1=\frac{\mn}{2y_2}+\mn\frac{m(y_1-m\log y_2)}{y_2(1+e^t)},$$
$$\ti{T}_{12}^2=-\ti{T}_{21}^2=\mn\frac{m}{1+e^t},$$
and the only nonzero components of $\db\ti{T}$ are
$$\de_{\ov{1}} \ti{T}_{12}^1=-\de_{\ov{1}} \ti{T}_{21}^1=-\frac{m}{2y_2(1+e^t)},$$
$$\de_{\ov{2}} \ti{T}_{12}^1=-\de_{\ov{2}} \ti{T}_{21}^1=\frac{1}{4y_2^2}+\frac{m(y_1-m\log y_2)+m^2}{2y_2^2(1+e^t)},$$
and $| \ov{\partial} \tilde{T} |_{\tilde{g}}\leq C$ follows easily. Next, by direct calculation, we have
$$\ti{\nabla}_{1}\ti{T}_{12}^1=O(e^{-t}),\quad \ti{\nabla}_{2}\ti{T}_{12}^1=O(1),\quad \ti{\nabla}_{1}\ti{T}_{12}^2=O(e^{-t}),\quad \ti{\nabla}_{2}\ti{T}_{12}^2=O(e^{-t}),$$
and $| \tilde{\nabla} \tilde{T}|_{\tilde{g}}\leq C$ follows.
To bound $|\widetilde{\mathrm{Rm}}|_{\ti{g}}$ we recall that
$\ti{R}_{i\ov{j}k}^{\ \ \ \ell}=-\de_{\ov{j}}\ti{\Gamma}_{ik}^\ell$.
Direct calculations show that $\ti{R}_{1\ov{1}1}^{\ \ \ 2}=0$, the terms
 $$\ti{R}_{2\ov{2}2}^{\ \ \ 2}, \ti{R}_{2\ov{2}2}^{\ \ \ 1}, \ti{R}_{1\ov{2}2}^{\ \ \ 1}, \ti{R}_{2\ov{1}2}^{\ \ \ 1}, \ti{R}_{2\ov{2}1}^{\ \ \ 1}$$  are of order $O(1)$ or better, while all other terms are of order $O(e^{-t})$ or better.
It follows  that $|\widetilde{\mathrm{Rm}}|_{\ti{g}}\leq C.$
This proves (ii).

As for (iii), this is the longest part of the calculation. Direct computations show that the components of the tensor $\ti{\nabla}_{\ov{p}} \ti{\nabla}_{\ov{q}} \ti{T}_{ij}^k$ with $k=2$ all vanish, while the components with $k=1$ are $O(e^{-t})$, and this gives $|\ov{\ti{\nabla}}\ov{ \ti{\nabla}} \tilde{T}|_{\ti{g}}\leq C$.   For $\tilde{\nabla} \ov{\tilde{\nabla}} \tilde{T}$, the  term  $\ti{\nabla}_{1} \ti{\nabla}_{\ov{1}} \ti{T}_{12}^2$ is $O(e^{-2t})$ and all other components are  $O(e^{-t})$ or better.
This gives $|\ti{\nabla}\ov{ \ti{\nabla}} \tilde{T}|_{\ti{g}}\leq C$.


For the derivative of curvature, we note that the terms
$$\ti{\nabla}_{\ov{2}} \ti{R}_{2\ov{2}2}^{\ \ \ 2}, \ti{\nabla}_{\ov{2}} \ti{R}_{2\ov{2}2}^{\ \ \ 1},   \ti{\nabla}_{\ov{1}} \ti{R}_{2\ov{2}2}^{\ \ \ 1}, \ti{\nabla}_{\ov{2}} \ti{R}_{1\ov{2}2}^{\ \ \ 1}, \ti{\nabla}_{\ov{2}} \ti{R}_{2\ov{1}2}^{\ \ \ 1}, \ti{\nabla}_{\ov{2}} \ti{R}_{2\ov{2}1}^{\ \ \ 1}$$
are $O(1)$ or better.  The terms
$$\ti{\nabla}_{\ov{1}} \ti{R}_{1\ov{1}1}^{\ \ \ 1}, \ti{\nabla}_{\ov{1}} \ti{R}_{1\ov{1}1}^{\ \ \ 2}, \ti{\nabla}_{\ov{2}} \ti{R}_{1\ov{1}1}^{\ \ \ 2}, \ti{\nabla}_{\ov{1}} \ti{R}_{2\ov{1}1}^{\ \ \ 2},  \ti{\nabla}_{\ov{1}} \ti{R}_{1\ov{2}1}^{\ \ \ 2},  \ti{\nabla}_{\ov{1}} \ti{R}_{1\ov{1}2}^{\ \ \ 2} $$
are  $O(e^{-2t})$ and all other components of the tensor $\ti{\nabla}_{\ov{p}} \ti{R}_{i\ov{j}k}^{\ \ \ \ell}$ are $O(e^{-t})$ or better.
 This finally gives $|\ti{\nabla}\widetilde{\mathrm{Rm}}|_{\ti{g}}\leq C$.
\end{proof}

We now wish to use these bounds, as in Proposition \ref{cal}, to obtain the estimate
\begin{equation}\label{c12}
|\ti{\nabla} g |_{\ti{g}}\leq C.
\end{equation}
A new complication that arises in this case is that the Christoffel symbols $\tilde{\Gamma}$ now \emph{do} depend on $t$.  From
$$\ddt{} \tilde{\Gamma}^i_{jk} = \tilde{g}^{i\ov{q}} \tilde{\nabla}_j \alpha'_{k\ov{q}},$$
and comparing to the computation in \cite{ShW}, we observe that there is one new term in the evolution of the quantity $S=|\ti{\nabla} g |_{g}^2$ of the form
\begin{equation}\label{newguy}
-2\mathrm{Re}\left(g_{i\ov{r}} g^{j\ov{u}}g^{k\ov{v}}\ti{g}^{i\ov{q}}\ti{\nabla}_j \alpha'_{k\ov{q}} \ov{\Psi^r_{uv}}\right),
\end{equation}
where, as in \cite{ShW}, $\Psi^r_{uv} = \Gamma^r_{uv} - \tilde{\Gamma}^r_{uv}$.

We claim that
\begin{equation}\label{alpha}
|\ti{\nabla}\alpha'|_{\ti{g}}\leq C.
\end{equation}
Indeed,  the only nonzero component of $\alpha'$ is $\alpha'_{2\ov{2}}=\frac{1}{2y_2^2}$, and so the only nonzero components of $\ti{\nabla}\alpha'$ are
\[\begin{split}
&\ti{\nabla}_1 \alpha'_{1\ov{2}}=-\ti{\Gamma}_{11}^2\alpha'_{2\ov{2}}=O(e^{-t}),\quad
\ti{\nabla}_1 \alpha'_{2\ov{2}}=-\ti{\Gamma}_{12}^2\alpha'_{2\ov{2}}=O(e^{-t}),\\
&\ti{\nabla}_2 \alpha'_{1\ov{2}}=-\ti{\Gamma}_{21}^2\alpha'_{2\ov{2}}=O(e^{-t}),\quad \ti{\nabla}_2 \alpha'_{2\ov{2}}=\de_{2}\alpha'_{2\ov{2}}-\ti{\Gamma}_{22}^2\alpha'_{2\ov{2}}=O(1),
\end{split}\]
thanks to \eqref{christ2}, and \eqref{alpha} now follows.  This shows that the new term (\ref{newguy}) is of the order $O(\sqrt{S})$ and hence is a harmless contribution.

Next, we would like to use \eqref{c12} to prove \eqref{c1new}, but again we have to deal with the fact that, unlike the case of $S_M$ with the Tricerri metric, the Christoffel symbols $\ti{\Gamma}$ of $\ti{g}$ now depend on $t$. However, we have the key relation \eqref{c13}.
\begin{proof}[Proof of Theorem \ref{mainprop} for the Inoue surfaces $S^+$ and $S^-$]
In the case of $S^+$, we use \eqref{c12}, \eqref{c13} and the fact that $\ti{g}\leq Cg_{\mathrm{V}}$ to obtain
$$|\nabla_{\mathrm{V}} g |_{g_{\mathrm{V}}}\leq |\ti{\nabla} g |_{g_{\mathrm{V}}}+C\leq C|\ti{\nabla} g |_{\ti{g}}+C\leq C,$$
as required. The case of $S^-$ follows by passing to a double cover.
\end{proof}

\section{Further Conjectures}\label{sectionconj}

We discuss in this section a number of conjectures and open problems concerning the Chern-Ricci flow.

\bigskip
\noindent
{\bf 1.} \ A natural conjecture, discussed in the introduction, is that the convergence result of Theorem \ref{maintheorem1} holds for \emph{all} initial Hermitian metrics $\omega_0$.  In fact, this would follow from Theorem \ref{maintheorem1} if one could show that every Hermitian metric on an Inoue surface belongs to the $\partial\ov{\partial}$-class of a Hermitian metric which is strongly flat along the leaves, a statement which can be reduced to solving a partial differential equation along the leaves.  A weaker conjecture, but still interesting in light of the results of \cite{TW, TW2, TWY}, would be that the convergence result of Theorem \ref{maintheorem1} holds for all \emph{Gauduchon} $\omega_0$.

\bigskip
\noindent
{\bf 2.} \ Another obvious conjecture, also mentioned in the introduction, is that the convergence  of $\omega(t)$ to $\omega_{\infty}$ in Theorem \ref{maintheorem1} holds in the $C^{\infty}$ topology.  A starting point would be to prove it in the more restrictive setting of Theorem \ref{mainprop}.

\bigskip
\noindent
{\bf 3.} \ There are higher-dimensional analogues of Inoue surfaces, constructed by Oeljeklaus-Toma \cite{OT}, and it is natural to conjecture that similar behavior occurs.  Similarly,  there are non-K\"ahler higher dimensional torus bundles over Riemann surfaces (see \cite[Example 3.4]{To}), and one would expect that at least some of the results of \cite{TWY} on elliptic bundles should generalize to these.

\bigskip
\noindent
{\bf 4.} \ It would be interesting to try to prove similar collapsing results on Hopf surfaces, the other family of Class VII surfaces with vanishing second Betti number.  In \cite{TW,TW2}, explicit solutions were given of the unnormalized Chern-Ricci flow $\ddt{} \omega = - \textrm{Ric}(\omega)$ on Hopf surfaces of the type $(\mathbb{C}^2 \setminus \{0 \})/ \sim$, where $(z_1, z_2) \sim (\alpha z_1, \beta z_2)$, for complex $\alpha, \beta$ with $|\alpha|=|\beta| \neq 1$.  In the Gromov-Hausdorff sense, the metrics collapse to a circle in finite time \cite{TW2}. Moreover, a uniform upper bound on the evolving metric was given in \cite{TW} starting at a metric in the $\partial \ov{\partial}$-class of the ``standard'' Hopf metric.  It would be desirable to understand the limiting behavior of the flow  starting at \emph{any} Hermitian metric, on \emph{any} Hopf surface.  In the case of Hopf surfaces where $|\alpha|\neq|\beta|$ or where $\sim$ is not of the type just described, no explicit solutions to the Chern-Ricci flow have even been constructed (see the discussion in \cite[Section 8]{TW}).

\bigskip
\noindent
{\bf 5.} \ Perhaps surprisingly, the behavior of the Chern-Ricci flow starting at a non-K\"ahler metric on the complex projective plane $\mathbb{CP}^2$ remains a mystery.  One can easily compute that the volume of the metric tends to zero in a finite time $T$, and the volume shrinks like $(T-t)$, and this suggests that the flow should collapse to a Riemann surface. This would be in stark contrast with what happens when starting at a K\"ahler metric on $\mathbb{CP}^2$, where the volume goes to zero like $(T-t)^2$ and the flow collapses in finite time to a point, thanks to Perelman's diameter bound (cf. \cite{SeT}).

\bigskip
\noindent
{\bf 6.} \ An optimistic conjecture would be that on any minimal Class VII surface with $b_2>0$, or on any Hopf surface, the solution $\omega(t)$ of the Chern-Ricci flow (which exists for a finite time $T$ and has volume tending to zero as $t \rightarrow T$  \cite{TW}) should converge in $C^{\infty}$ to a non-closed nonnegative $(1,1)$ form $\alpha$.  Moreover, taking the tangent distribution given by the kernel of $\alpha$ and taking iterated Lie brackets of it, one should obtain an integrable three-dimensional distribution.  Even more optimistically, one might hope that a leaf of this distribution be a sphere in a global spherical shell \cite{Ka}, which is conjectured to exist \cite{Na}. This behavior is exactly what happens for the explicit solutions constructed in \cite{TW2} on the standard Hopf surfaces as in item 4. above.


\begin{thebibliography}{99}
\bibitem{Bo} Bogomolov, F.A. {\em Surfaces of class ${\rm VII}\sb{0}$ and affine geometry}, Izv. Akad. Nauk SSSR Ser. Mat. {\bf 46} (1982), no. 4, 710--761.
\bibitem{Br} Brunella, M. {\em Subharmonic variation of the leafwise Poincar\'e metric}, Invent. Math. {\bf 152} (2003), no. 1, 119--148.
\bibitem{Br2}  Brunella, M. {\em A characterization of Inoue surfaces}, Comment. Math. Helv. {\bf 88} (2013), no. 4, 859--874.
\bibitem{Ca} Candel, A. {\em Uniformization of surface laminations}, Ann. Sci. \'Ecole Norm. Sup. (4) {\bf 26} (1993), no. 4, 489--516.
\bibitem{CY} Cheng, S.Y., Yau, S.-T. {\em Differential equations on Riemannian manifolds and their geometric applications}, Comm. Pure Appl. Math. {\bf 28} (1975), no. 3, 333--354.
\bibitem{FZ} Fong, F.T.-H., Zhang, Z. {\em The collapsing rate of the K\"ahler-Ricci flow with regular infinite time singularity}, arXiv:1202.3199, to appear in J. Reine Angew. Math.
\bibitem{Gh} Ghys, \'E. {\em Sur l'uniformisation des laminations paraboliques}, in {\em Integrable systems and foliations (Montpellier, 1995)}, 73--91, Progr. Math., 145, Birkh\"auser, Boston, MA, 1997.
\bibitem{G} Gill, M. {\em Convergence of the parabolic complex Monge-Amp\`ere equation on compact Hermitian manifolds}, Comm. Anal. Geom. {\bf 19} (2011), no. 2, 277--303.
\bibitem{G2} Gill, M. {\em Collapsing of products along the K\"ahler-Ricci flow}, Trans. Amer. Math. Soc. {\bf 366} (2014), no. 7, 3907--3924.
\bibitem{G3} Gill, M. {\em The Chern-Ricci flow on smooth minimal models of general type}, arXiv:1307.0066.
\bibitem{GS} Gill, M., Smith, D.J. {\em The behavior of Chern scalar curvature under Chern-Ricci flow}, arXiv:1311.6534.
\bibitem{In} Inoue, M. {\em On surfaces of Class $VII_0$}, Invent. Math. {\bf 24} (1974), 269--310.
\bibitem{Ka}  Kato, Ma. {\em Compact complex manifolds containing ``global'' spherical shells. I}, in {\em Proceedings of the International Symposium on Algebraic Geometry (Kyoto Univ., Kyoto, 1977)}, 45-–84, Kinokuniya Book Store, Tokyo, 1978.
\bibitem{Ko} Kodaira, K. {\em On the structure of compact complex analytic surfaces, II}, Amer. J. Math. {\bf 88} (1966), no. 3, 682--721.
\bibitem{L} Lauret, J. {\em Curvature flows for almost-hermitian Lie groups}, arXiv:1306.5931, to appear in Trans. Amer. Math. Soc.
\bibitem{LR} Lauret, J., Rodr\'iguez Valencia, E.A. {\em On the Chern-Ricci flow and its solitons for Lie groups}, arXiv:1311.0832, to appear in Math. Nachr.
\bibitem{LYZ} Li, J., Yau, S.-T., Zheng, F. {\em On projectively flat Hermitian manifolds}, Comm. Anal. Geom. {\bf 2} (1994), 103--109.
\bibitem{Na} Nakamura, I. {\em On surfaces of class $\rm VII\sb 0$ with curves}, Invent. Math. {\bf 78} (1984), no. 3, 393--443.
\bibitem{Ni} Nie, X. {\em Regularity of a complex Monge-Amp\`ere equation on Hermitian manifolds}, Comm. Anal. Geom. {\bf 22} (2014), no. 5, 833--856.
\bibitem{OT} Oeljeklaus, K., Toma, M. {\em Non-K\"ahler compact complex manifolds associated to number fields}, Ann. Inst. Fourier (Grenoble) {\bf 55} (2005), no. 1, 161--171.
\bibitem{PSS} Phong, D.H., Sesum, N., Sturm, J. {\em Multiplier ideal sheaves and the K\"ahler-Ricci flow}, Comm. Anal. Geom. {\bf 15} (2007), no. 3, 613--632.
\bibitem{PS} Phong, D.H., Sturm, J. {\em The Dirichlet problem for degenerate complex Monge-Ampere equations}, Comm. Anal. Geom. {\bf 18} (2010), no. 1, 145--170.
\bibitem{SeT} Sesum, N., Tian, G. {\em Bounding scalar curvature and diameter along the K\"ahler Ricci flow (after Perelman)}, J. Inst. Math. Jussieu {\bf 7} (2008), no. 3, 575--587.
\bibitem{ShW} Sherman, M., Weinkove, B. {\em Local Calabi and curvature estimates for the Chern-Ricci flow}, New York J. Math. {\bf 19} (2013), 565--582.
\bibitem{ST} Song, J., Tian, G. {\em The K\"ahler-Ricci flow on surfaces of positive Kodaira dimension}, Invent. Math. {\bf 170} (2007), no. 3, 609--653.
\bibitem{ST2} Song, J., Tian, G. {\em Canonical measures and K\"ahler-Ricci flow}, J. Amer. Math. Soc. {\bf 25} (2012), no. 2, 303--353.
\bibitem{ST3} Song, J., Tian, G. {\em Bounding scalar curvature for global solutions of the K\"ahler-Ricci flow}, arXiv:1111.5681.
\bibitem{SW} Song, J., Weinkove, B. {\em An introduction to the K\"ahler-Ricci flow}, in {\em An introduction to the K\"ahler-Ricci flow}, 89--188, Lecture Notes in Math. {\bf 2086}, Springer, Cham., 2013.
\bibitem{STi} Streets, J., Tian, G. {\em A parabolic flow of pluriclosed metrics}, Int. Math. Res. Not. {\bf 2010} (2010), no. 16, 3103--3133
\bibitem{T0} Teleman, A. {\em Projectively flat surfaces and Bogomolov's theorem on class $VII_{0}$-surfaces},  Int. J. Math. {\bf 5} (1994), 253--264.
\bibitem{To} Tosatti, V. {\em Non-K\"ahler Calabi-Yau manifolds}, arXiv:1401.4797, to appear in Contemp. Math.
\bibitem{TW} Tosatti, V., Weinkove, B. \emph{On the evolution of a Hermitian metric by its Chern-Ricci form}, J. Differential Geom. {\bf 99} (2015), no.1, 125--163.
\bibitem{TW2} Tosatti, V., Weinkove, B. {\em The Chern-Ricci flow on complex surfaces},  Compos. Math. {\bf 149} (2013), no. 12, 2101--2138.
\bibitem{TWY} Tosatti, V., Weinkove, B., Yang, X. {\em Collapsing of the Chern-Ricci flow on elliptic surfaces}, arXiv:1302.6545, to appear in Math. Ann.
\bibitem{TWY2} Tosatti, V., Weinkove, B., Yang, X. {\em The K\"ahler-Ricci flow, Ricci-flat metrics and collapsing limits}, arXiv:1408.0161.
\bibitem{Tr} Tricerri, F. {\em Some examples of locally conformal K\"ahler manifolds}, Rend. Sem. Mat. Univ. Politec. Torino {\bf 40} (1982), no. 1, 81--92.
\bibitem{Va} Vaisman, I. {\em Non-K\"ahler metrics on geometric complex surfaces}, Rend. Sem. Mat. Univ. Politec. Torino {\bf 45} (1987), no. 3, 117--123.
\bibitem{Wa} Wall, C.T.C. {\em Geometric structures on compact complex analytic surfaces}, Topology {\bf 25} (1986), no. 2, 119--153.
\bibitem{Ya}Yau, S.-T. {\em A general Schwarz lemma for K\"ahler manifolds}, Amer. J. Math. {\bf 100} (1978), no. 1, 197--203.

\end{thebibliography}
\end{document}